\documentclass[12pt]{amsart}
\usepackage{mathrsfs}
\usepackage{latexsym,amsthm,amssymb,amsmath,amsfonts,graphicx,epstopdf,subfigure,float}
\usepackage{bbm}
\usepackage{pifont}
\usepackage{cite}
\usepackage{multirow}
\usepackage{appendix}
\usepackage{tikz,xcolor}

\usetikzlibrary{decorations.markings}
\usetikzlibrary{intersections}
\usetikzlibrary{calc}

\newtheorem{theorem}{Theorem}[section]
\newtheorem{thm}[theorem]{Theorem}

\newtheorem{lemma}{Lemma}[section]

\newtheorem{defi}{Definition}[section]

\numberwithin{equation}{section}

\def\R{\mathbb R}

 \usepackage{color}

\begin{document}

\title{Locally measure preserving property of bi-Lipschitz maps between Moran sets}

\author{Liang-yi Huang}
\address{College of Science, Wuhan University of Science and Technology, Wuhan, 430070, China}
\email{liangyihuang@wust.edu.cn}

\author{Shishuang Liu$^*$}
\address{Department of Mathematics and Statistics, Central China Normal University, Wuhan, 430079, China}
\email{shishuangliu@mails.ccnu.edu.cn}


\date{\today}

\thanks{{\bf 2010 Mathematics Subject Classification:}  28A80, 26A16\\
 {\indent\bf Key words and phrases:}\ Moran set, Measure preserving property }

\thanks{* Corresponding author.}

\begin{abstract}
In literature it is shown that bi-Lipschitz maps between self-similar sets or self-affine sets
enjoy a locally measure preserving property, namely, if $f:(E,\mu)\to (F,\nu)$ is a bi-Lipschitz map, then
the Radon-Nykodym derivative   $df^*\nu/d\mu$ is a constant function
on a subset $E'\subset E$ with $\mu(E')>0$, where $f^*\nu(\cdot)=\nu(f(\cdot))$.
Indeed, this measure preserving property plays an important role in Lipschitz classification of fractal sets. In this paper, we show that  such measure preserving property also holds for bi-Lipschitz maps between two Moran sets in a certain class.
\end{abstract}
\maketitle

\section{\textbf{Introduction}}
Let $(E, d_E, \mu)$ and $(F,d_F,\nu)$  be two metric spaces with  finite Borel measures $\mu$ and $\nu$, respectively.
We say these two spaces enjoy a \emph{universal locally measure preserving property},
if for any  bi-Lipschitz map   $f:E\to F$,
the Radon-Nykodym derivative
$$
\frac{df^*\nu}{d\mu}=\text{constant}
$$
on a subset $E'\subset E$ with $\mu(E')>0$, where
$$f^*\nu(\cdot)=\nu(f(\cdot)).$$

 Cooper and Pignataro  \cite{CP88}  proved that for  two self-similar sets $E,F\subset  \mathbb R$ satisfy the strong separation condition equipped with the Hausdorff measure (positive in its Hausdorff dimension), all monotone bi-Lipschitz maps
 enjoy the locally measure preserving bi-Lipschitz.  Later, Falconer and Marsh \cite{FM92} and Xi and Ruan \cite{XR08} showed that the same result holds for $E,F\subset \R^d$ without the orientation preserving assumption. Rao, Yang and Zhang \cite{RYZ21} showed the universal locally measure preserving property for Bedford-McMullen carpets equipped with the uniform Bernoulli measures.

This locally measure preserving property turns out to play a very important role in
Lipschitz classification  of self-similar sets and self-affine sets, see
 \cite{ML10,FM92,RRX06,RRW12,LL13,XY21,XR08,RWX14,RRW13,X10}.

Recall that two metric spaces $(X, d_X)$ and $(Y, d_Y)$ are said to be \emph{Lipschitz equivalent},
denoted by $(X,d_X)\sim (Y,d_Y)$, if there exists a map
$f:~X\rightarrow Y$ which is bi-Lipschitz, that is, there exists a constant $C>0$ such that
for all $x,y\in X$,
$$C^{-1}d_X(x,y)\leq d_Y(f(x),f(y))\leq C d_X(x,y).$$

The main goal of the present paper is to show that the universal locally measure preserving property
holds for a class of homogeneous Moran sets equipped with the uniform Bernoulli measures.

  Let us recall the definition of Moran set. Let $\{n_k\}_{k\ge 1}$ be a sequence of positive integers with $n_k\ge 2$.
Set $\Sigma^0=\emptyset$, and
$\Sigma^k=\prod_{j=1}^k \{0,\dots, n_{j}-1\}$ for $k\ge 1$. Denote  $\Sigma^*=\bigcup_{k\ge 0}\Sigma^k$.
For $\boldsymbol{\sigma}=(\sigma_1,\dots,\sigma_k)$ and $\boldsymbol{\tau}=(\tau_1,\dots,\tau_{s})$, let $\boldsymbol{\sigma}*\boldsymbol{\tau}=(\sigma_1,\dots,\sigma_k,\tau_1,\dots,\tau_{s})$ be their concatenation.
Let $J$ be a closed interval. A collection of closed subintervals $\{J_{\boldsymbol{\sigma}}: \boldsymbol{\sigma}\in \Sigma^*\}$ of $J$ is said to have  a \emph{Moran structure}, if\\
(i) $J_\emptyset=J$;\\
(ii)  for any $k\ge 1$, $\boldsymbol{\sigma}\in \Sigma^{k-1}$ and $i\in\{0,\dots,  n_k-1\}$, $J_{\boldsymbol{\sigma}*i}$ are subintervals of $J_{\boldsymbol{\sigma}}$ with disjoint interiors.

Let
\begin{equation*}
{\mathbf F}_k=\bigcup\limits_{\boldsymbol{\sigma}\in \Sigma^k}J_{\boldsymbol{\sigma}}, \quad F=\bigcap\limits_{k\ge 0}{\mathbf F}_k.
\end{equation*}
Then $F$ is a compact set and we call $F$ a \emph{Moran set} with   parameters $(J,\{n_k\}_{k\ge 1})$. Let $\boldsymbol{\sigma}\in \Sigma^k$.  We call $J_{\boldsymbol{\sigma}}$ a \emph{basic interval} of rank $k$, and   $F_{\boldsymbol{\sigma}}=J_{\boldsymbol{\sigma}}\cap F$ a \emph{cylinder} of rank $k$.

 We call $F$ a \emph{homogeneous Moran set} with parameters
 $(J,\{n_k\}_{k\ge 1},\{r_k\}_{k\ge 1})$, if in addition, $F$ satisfies that\\
(iii) for any $k\ge 1$, $\boldsymbol{\sigma}\in\Sigma^{k-1}$ and $0\leq i\leq n_k-1$,
$\frac{|J_{\boldsymbol{\sigma}*i}|}{|J_{\boldsymbol{\sigma}}|}=r_k,$ where $|A|$ denotes the diameter of set $A$.

We denote the collection of  such Moran sets by
\begin{equation}\label{eq:Moran-1}
{\mathscr F}(J,\{n_k\}_{k\ge 1},\{r_k\}_{k\ge 1}),
\end{equation}
see \cite{FWW99,HRWW00,W01}.
 Let $F$ be a homogeneous Moran set in \eqref{eq:Moran-1},  we define  $\Delta_1=|J|(1-n_1r_1)$,
$$\Delta_k=|J|r_1\dots r_{k-1}(1-n_kr_k), \ \forall k\geq 2$$
and we call $\Delta_k$ \emph{the total gap of rank $k$} of $F$.

\medskip

For $E, F\subset \R$, we denote $\text{dist}(E,F)=\inf\{|x-y|;~x\in E, y\in F\}$. 

\begin{defi}\emph{
We say a Moran set $F$ satisfies the \emph{weak separation condition} (WSC), if there is a constant $\eta_0>0$ such that for any $k\ge 0$, $\boldsymbol{\sigma}\in\Sigma^k$ and $  i\neq j$,
\begin{equation*}
 \text{either }J_{\boldsymbol{\sigma}*i}\cap J_{\boldsymbol{\sigma}*j}\neq\emptyset \text{ or } \text{dist}(J_{\boldsymbol{\sigma}*i},J_{\boldsymbol{\sigma}*j})>\eta_0\left (|J_{\boldsymbol{\sigma}}|-\sum_{i=0}^{n_{k+1}-1} |J_{\boldsymbol{\sigma}*i}|\right ).
\end{equation*}
}
\end{defi}


Our main result is the following.

\begin{thm}\label{thm:measure-1}
Let $E,F \in {\mathscr F}([0,1],\{n_k\}_{k\ge 1},\{r_k\}_{k\ge 1})$. Let $f:E\to F$ be a bi-Lipschitz map. Suppose that

\emph{(i)} $\sup_{k\ge 1}n_k<\infty$;

\emph{(ii)} $\gamma=\inf_{k\geq 1}\frac{\Delta_{k}}{\Delta_{k+1}}>1$;

\emph{(iii)} both $E$ and $F$ satisfy the \emph{WSC}.\\
Then for any cylinder $E_{\boldsymbol{\sigma}}$ of $E$,  there exists a sub-cylinder $E_{\boldsymbol{\sigma\tau}}$ of $E_{\boldsymbol{\sigma}}$ such that $df^*\nu/d\mu$ is a constant function on $E_{\boldsymbol{\sigma\tau}}$, equivalently
$$
\frac{\nu(f(\mathbf B))}{\mu(\mathbf B)}=\frac{\nu(f(E_{\boldsymbol{\sigma\tau}}))}{\mu( E_{\boldsymbol{\sigma\tau}})}
$$
for any Borel set $\mathbf B\subset E_{\boldsymbol{\sigma\tau}}$.
\end{thm}

\section{\textbf{Proof of Theorem \ref{thm:measure-1}} }

Measure preserving properties of bi-Lipschitz maps of self-similar sets and self-affine sets
have been studied in \cite{FM92,HRWX24,RYZ21,XR08 }. In this section, we investigate such property related to homogenous Moran sets.

In this section, unless otherwise specified, we always assume that $E,F$ are two homogeneous Moran sets in  ${\mathscr F}([0,1],\{n_k\}_{k\ge 1}, \{r_k\}_{k\ge 1})$.
Let $\mu$ be the Borel probability measure supported on $E$   such that for each $k$-th cylinder $E_{\boldsymbol{\sigma}}$,
$$\mu(E_{\boldsymbol{\sigma}}) =\frac{1}{n_1\dots n_k}.$$
We call $\mu$ the \emph{uniform Benoulli measure} of $E$. Let $\nu$ be the uniform Bernoulli measure of $F$.

\begin{thm}\label{thm-equivalence}
Let $f:E\to F$ be a bi-Lipschitz map, then $f^*\nu$ is equivalent to $\mu$, that is, there exists $\alpha>0$ such that
$$
  \alpha^{-1}\mu(\mathbf B)\le \nu(f(\mathbf B))\le \alpha \mu(\mathbf B),
$$
for any Borel set $\mathbf B\subset E$.
\end{thm}

\begin{proof} Let $C$ be a  Lipschitz constant of $f$.

First, we claim that there exists $\alpha>0$ such that for each cylinder $E_{\boldsymbol{\sigma}}$, it holds that
$$
\nu(f(E_{\boldsymbol{\sigma}}))\le \alpha \mu(E_{\boldsymbol{\sigma}}).
$$
Denote $k=|\boldsymbol{\sigma}|$, then $\mu( E_{\boldsymbol{\sigma}})=  \frac{1}{n_1\dots n_k}$
and
\begin{equation*}
  \text{diam}\, f( E_{\boldsymbol{\sigma}})\le C (\text{diam}\,  E_{\boldsymbol{\sigma}})= Cr_1\dots r_k.
\end{equation*}
 Therefore, there are at most $C+2$ cylinders of rank $k$ of $F$ intersecting $f( E_{\boldsymbol{\sigma}})$, so we have
\begin{equation*}
  \nu(f( E_{\boldsymbol{\sigma}}))\le \frac{C+2}{n_1\dots n_k}=(C+2)\mu( E_{\boldsymbol{\sigma}}).
\end{equation*}
The claim is proved.

Since the Borel $\sigma$-algebra of $E$ can be generated by
 the collection of cylinders,
it follows that $\nu(f(\mathbf B))\le \alpha \mu(\mathbf B)$ holds for any Borel set $\mathbf B\subset E$. By changing the role of $E$ and $F$, we can get the other side inequality. This completes the proof of the theorem.
\end{proof}

\begin{defi}[$\delta$-connected]
\emph{ A compact set $X\subset \mathbb{R}$ is said to be \emph{$\delta$-connected}, if for any $x,y\in X$, there exists a sequence $\{x=z_1,\dots,z_n=y\}\subset X$ such that $|z_i-z_{i+1}|\le \delta$ holds for $1\le i\le n-1$. The sequence above is called a \emph{$\delta$-chain} connecting $x$ and $y$.}
\end{defi}

Recall that ${\mathbf E}_k$ is the union of basic intervals of rank $k$ of $E$;  we call  ${\mathbf E}_k$  the $k$-th approximation of $E$. We
 use ${\mathcal E}_k$ to denote the collection of connected component of ${\mathbf E}_k$.
Similarly, we define ${\mathbf F}_k$ and ${\mathcal F}_k$.
Recall that $\Delta_k$
 is the total gap of rank $k$.

\medskip

 The following lemma is a variation of Lemma 3.2 in Falconer and Marsh \cite{FM92}.

\begin{lemma}\label{decomposition-lem-1}
Let $f:E\to F$ be a bi-Lipschitz map with Lipschitz constant $C$. Suppose that both $E$ and $F$ satisfy the \emph{WSC} and
$$
\displaystyle \gamma=\inf_{k\geq 1}\frac{\Delta_{k}}{\Delta_{k+1}}>1.
$$
Then there exists a positive integer $p_0$ such that,
for any $k\ge p_0$ and $I\in {\mathcal E}_k$, there exist  ${J^*}\in {\mathcal F}_{k-p_0}$ and $\{J_j\}_{j=1}^h\subset {\mathcal F}_{k+p_0}$ such that
$\{J_j\}_{j=1}^h\subset J^*$ and
\begin{equation*}
  f(I\cap E)=\bigcup_{j=1}^{h}(J_j\cap F).
\end{equation*}
\end{lemma}

\begin{proof} Let $\eta_0$ be a common constant in the weak separation condition of $E$ and $F$. Let $p_0$ be the smallest integer satisfying
 $(\gamma-1)\gamma^{p_0}>\frac{2C}{\eta_0}$.

Let $U, U'$ be two distinct elements of ${\mathcal E}_k$. By the weak separation condition,
we have
$$
\text{dist}\, (U,U')\geq \min\{\eta_0 \Delta_j; 1\leq j\leq k\}=\eta_0\Delta_k.
$$
On the other hand, for any $U\in {\mathcal E}_k$, $U\cap E$ is $(\frac{2}{\gamma-1}\Delta_k)$-connected since the maximal gap of $U\cap E$ is no larger than
$
2\sum_{j=1}^{\infty}\Delta_{k+j}\le 2\sum_{j=1}^{\infty}\frac{\Delta_k}{\gamma^j}=\frac{2}{\gamma-1}\Delta_{k}.
$
Moreover, these two facts are also true for ${\mathcal F}_k$.

Pick any $J\in{\mathcal F}_{k+p_0}$. We claim that
$$\text{either $(J\cap F)\subset f(I\cap E)$},\ \text{or $(J\cap F)\cap f(I\cap E)=\emptyset$}.$$
  First, $J\cap F$ is $(\frac{2}{\gamma-1}\Delta_{k+p_0})$-connected  implies
$f^{-1}(J\cap F)$ is  $(\frac{2C}{\gamma-1}\Delta_{k+p_0})$-connected.
By our choice of $p_0$, we have
$(\frac{2C}{\gamma-1}\Delta_{k+p_0})\leq (\frac{2C}{\gamma-1}\gamma^{-p_0}\Delta_{k})  <\eta_0 \Delta_k$,
which implies that either $f^{-1}(J\cap F)\subset I$
or $f^{-1}(J\cap F)\cap I=\emptyset.$
The claim is proved.

Applying the above claim to the map $f^{-1}:F\to E$, there exists $J\in{\mathcal F}_{k-p_0}$ such that $ f(I\cap E)\subset J\cap F$. This finishes the proof of the lemma.
\end{proof}

 We denote by $\mathcal{C}_{E}=\bigcup_{k\ge 0}{\mathcal E}_k$, where $ {\mathcal E}_{0}=\{[0,1]\}$ by convention.
Let $U,U'\in \mathcal{C}_{E}$, we say $U'$ is an offspring of $U$ if $U'\subset U$, and it is called a direct offspring of $U$ if the rank of $U$ equals the rank of $U'$ plus $1$.

\medskip

\begin{proof}[\textbf{Proof of Theorem \ref{thm:measure-1}.}] Let $E_{\boldsymbol\sigma}$ be a cylinder of $E$.
Let $U_0$ be a component in ${\mathcal C}_E$ such that $U_0\cap E\subset E_{\boldsymbol \sigma}$. Set
$$
{\mathcal C}_{E, U_0}=\{V\in {\mathcal C}_{E};~ V\subset U_0\}.
$$
Define
$$\varphi(U)=\frac{\nu(f(U\cap E))}{\mu(U)},\ \text{where $U\in\mathcal{C}_{E, U_0}$.}$$
By Theorem \ref{thm-equivalence}, we have
$$\chi=\sup_{U\in\mathcal{C}_{E,U_0}}\varphi(U)<\infty.$$

Since each $U\in \mathcal{C}_{E,U_0}$ contains a basic interval and vice versa, the theorem holds if and only if there exists $U\in \mathcal{C}_{E,U_0}$ such that $f$ is measure preserving on $U\cap E$.

Suppose on the contrary that $f|_{U\cap E}$ is not measure preserving for any $U\in\mathcal{C}_{E,U_0}$. Then we can get that $\varphi(U)<\chi$ for all $U\in\mathcal{C}_{E,U_0}$, for otherwise, $f|_{U\cap E}$ is measure preserving for these $U$ such that $\varphi(U)$ attains $\chi$.

Denote $\beta=\sup_{k\ge 1}n_k$.
Set $\varepsilon=\frac{1}{2(1+8\beta^{2p_0+5})}$, where $p_0$ is the constant in Lemma \ref{decomposition-lem-1}. By the definition of $\chi$, there is $U\in \mathcal{C}_{E, U_0}$ such that
$$\chi(1-\varepsilon)<\varphi(U)<\chi.$$

Let $k$ be the rank of $U$ and let $U_1,\dots,U_{\ell}$ be the direct offsprings of $U$, we have either $\varphi(U)=\varphi(U_j)$ for all $1\le j\le \ell$, or there exists
$j_0\in\{1,\dots,\ell\}$ such that $\varphi(U_{j_0})>\varphi(U)$. Since $f|_{U\cap E}$ is not measure preserving, there exists a sequence $U=V_0,V_1,\dots,V_{t-1},V_t$ in $\mathcal{C}_{E, U_0}$ such that $V_j$ is directed offspring of $V_{j-1}$ for $j=1,\dots, t$ and $\varphi(U)=\varphi(V_0)=\varphi(V_1)=\dots=\varphi(V_{t-1})\neq \varphi(V_t)$. Without loss of generality, we may assume that $\varphi(V_t)>\varphi(V_{t-1})$ and we can replace $U$ by $V_{t-1}$ to start our discussion.

Next we estimate $\varphi(U_{j_0})=\frac{\nu(f(U_{j_0}\cap E))}{\mu(U_{j_0})}.$ Notice that $\varphi(U)=\frac{\nu(f(U\cap E))}{\mu(U)},$ we have
\begin{equation}\label{varphiU}
 \varphi(U_{j_0})=\frac{\nu(f(U_{j_0}\cap E))}{\nu(f(U\cap E))}\cdot\frac{\mu(U)}{\mu(U_{j_0})}\cdot\varphi(U).
\end{equation}
By Lemma \ref{decomposition-lem-1}, $f(U\cap E)$ can be decomposed into
$$
f(U\cap E)=\bigcup_{j=1}^{h}({J}_j\cap F),
$$
where $ \{J_j\}_{j=1}^{h}\subset  {\mathcal F}_{k+p_0}$, and are offsprings of $J^*\in {\mathcal F}_{k-p_0}$. Since each $J\in {\mathcal F}_{k-p_0}$ contains at most $2n_{k-p_0}$  basic intervals  of rank $k-p_0$, we have
$$h\le (2n_{k-p_{0}})\prod_{j=1}^{2p_0}n_{k-p_{0}+j}\le 2\beta^{2p_0+1}.$$
 Applying Lemma \ref{decomposition-lem-1} to $U_{j_0}$, there exists $\{ {I}_j\}_{j=1}^{s}\subset {\mathcal F}_{k+1+p_0}$ such that
$f(U_{j_0}\cap E)=\bigcup_{j=1}^{s}({I}_j\cap F)$.

For any connected component $Q$ of rank $k$, denote by $\#_{k}{Q}$ the number of $k$-th basic intervals contained in $Q$. We have
$$\nu(f(U\cap E))=\sum_{j=1}^{h}\nu({J}_{j}\cap F)=\sum_{j=1}^{h}\frac{\#_{(k+p_0)} {J}_j}{n_1\dots n_{k+p_0}},$$
and
$$\nu(f(U_{j_0}\cap E))=\sum_{j=1}^{s}\nu( {I}_{j}\cap F)=\sum_{j=1}^{s}\frac{\#_{(k+p_0+1)} {I}_j}{n_1\dots n_{k+p_0}n_{k+p_0+1}}.$$
Substituting these two relations to \eqref{varphiU}, we obtain
$$
\frac{ \varphi(U_{j_0})}{ \varphi(U)}=\frac{\sum_{j=1}^{s} \#_{(k+p_0+1)} {I}_j  }{\sum_{j=1}^{h} \#_{(k+p_0)} {J}_j  }\cdot\frac{1}{n_{k+p_0+1}}\cdot\frac{\#_{k}U}{\#_{(k+1)} U_{j_0}}\cdot\frac{1}{n_{k+1}}.
$$
Notice that the denominator of the right hand side is no larger than $8\beta^{2p_0+5}$,
 by $\varphi(U_{j_0})>\varphi(U)$, we obtain
$$
  \varphi(U_{j_0})>(1+\frac{1}{8\beta^{2p_0+5}})\varphi(U)>\chi,
$$
which contradicts the definition of $\chi$. This completes the proof of the Theorem.
\end{proof}

\end{document}